\documentclass[11pt]{article}
\usepackage[utf8]{inputenc}
\usepackage{enumerate,amsthm,amssymb,xcolor,hyperref,comment}
\usepackage{thmtools}
\usepackage{thm-restate}
\usepackage{authblk}
\usepackage[top=1in, left=1in, bottom=1in, right=1in]{geometry}

\usepackage[leqno]{amsmath}

\makeatletter
\newcommand{\leqnomode}{\tagsleft@true}
\newcommand{\reqnomode}{\tagsleft@false}
\makeatother
\declaretheorem{theorem}
\declaretheorem[sibling=theorem]{lemma}
\declaretheorem[sibling=theorem]{Corollary}
\newcommand{\sset}[1]{\left\{#1\right\}}

\title{List-three-coloring $ P_t $-free graphs with no induced  1-subdivision of $ K_{1,s} $}
\author[1]{Maria Chudnovsky\footnote{Partially supported
		by NSF grant DMS-1763817 and U.S. Army Research Office grant W911NF-16-1-0404.}}
\author[1]{Sophie Spirkl\footnote{This material is based upon work supported by the National Science Foundation under Award No. DMS-1802201}}
\author[2]{Mingxian Zhong}

\affil[1]{\emph{\small Princeton University, Princeton, NJ 08544}}

\affil[2]{\emph{\small
Lehman College and the Graduate Center, City University of New York, NY 10468}}
\date{\today}
\begin{document}
\maketitle

\begin{abstract} 
Let $s$ and $t$ be positive integers. We use $P_t$ to denote the path with $t$ vertices and $K_{1,s}$ to denote the complete bipartite graph with parts of size $1$ and $s$ respectively. The one-subdivision of $K_{1,s}$ is obtained by replacing every edge $\{u,v\}$ of $K_{1,s}$ by two edges $\{u,w\}$ and $\{v,w\}$ with a new vertex $w$. In this paper, we give a polynomial-time algorithm for the list-three-coloring problem restricted to the class of $P_t$-free graph with no induced 1-subdivision of $K_{1,s}$.

\end{abstract}

	\section{Introduction}

All graphs in this paper are finite and simple.
We use $[k]$ to denote the set $\sset{1, \dots, k}$. Let $G$ be a graph. A
\emph{$k$-coloring} of $G$ is a function
$f:V(G) \rightarrow [k]$ such that for every edge 
$uv \in E(G)$,
$f(u) \neq f(v)$, and $G$ is \emph{$k$-colorable} if $G$ has a
$k$-coloring. The \textsc{$k$-coloring problem} is the problem of
deciding, given a graph $G$, if $G$ is $k$-colorable. This
problem is well-known to be $NP$-hard for all $k \geq 3$.

A function $L: V(G) \rightarrow 2^{[k]}$ that assigns a subset of
$[k]$ to each vertex of a graph $G$ is a \emph{$k$-list assignment}
for $G$. For a $k$-list assignment $L$, a function
$f: V(G) \rightarrow [k]$ is a 
\emph{coloring of $(G,L)$} if $f$ is a
$k$-coloring of $G$ and $f(v) \in L(v)$ for all $v \in V(G)$. 
We say that a graph $G$ is  \emph{$L$-colorable},
and that the pair $(G,L)$ is {\em colorable},  if $(G,L)$ has a coloring.
The \textsc{list-$k$ coloring problem} is the problem of deciding, given a
graph $G$ and a $k$-list assignment $L$, if $(G,L)$ is 
colorable. Since this generalizes the $k$-coloring problem, it is
also $NP$-hard for all $k \geq 3$.


 We denote by $P_t$ the path with $t$ vertices and we use $K_{1,s}$ to denote the complete bipartite graph with parts of size $1$ and $s$ respectively. The one-subdivision of $K_{1,s}$ is obtained by replacing every edge $\{u,v\}$ of $K_{1,s}$ by two edges $\{u,w\}$ and $\{v,w\}$ with a new vertex $w$. For a set $\mathcal{H}$ of graphs, a graph $G$ is $\mathcal{H}$-free if no element of $ \mathcal{H} $ is an induced subgraph of $G$. If $\mathcal{H}=\{H\}$, we say that $G$ is $H$-free.  In this paper, we use the terms
``polynomial time'' and ``polynomial size'' to mean ``polynomial in
$|V(G)|$'', where $G$ is the input graph.
Since the  \textsc{$k$-coloring problem} and the \textsc{list-$k$ coloring problem} are $NP$-hard for $k \geq 3$, their
restrictions to $H$-free graphs, for various $H$,  have been extensively 
studied. In particular, the following is known:

\begin{theorem}[\cite{gps}] Let $H$ be a (fixed) graph, and let $k>2$. If
	the \textsc{$k$-coloring problem} can be solved in polynomial time when restricted to the class of $H$-free graphs, then every connected component of $H$ is a path.
\end{theorem}

Thus if we assume that $H$ is connected, then the question of determining the 
complexity of $k$-coloring $H$-free graph is reduced to 
studying the complexity of coloring graphs with 
certain induced paths excluded,
and a significant body of work has been produced on this topic.
Below we list a few such results.

\begin{theorem}[\cite{c1}] \label{3colP7}
	The \textsc{3-coloring problem} can be
	solved in polynomial time for the class of $P_7$-free graphs.
\end{theorem}
\begin{theorem} [\cite{4p6}]
	The \textsc{4-coloring problem} can be
	solved in polynomial time for the class of $P_6$-free graphs.
\end{theorem}
\begin{theorem}[\cite{hoang}] The \textsc{$k$-coloring problem} can be
	solved in polynomial time for the class of $P_5$-free graphs.
\end{theorem}

\begin{theorem}[\cite{huang}] The \textsc{4-coloring problem} is
	$NP$-complete for the class of $P_7$-free graphs.
\end{theorem}

\begin{theorem}[\cite{huang}] For all $k \geq 5$, the
	\textsc{$k$-coloring problem} is $NP$-complete for the class of
	$P_6$-free graphs.
\end{theorem}

The only case for which the complexity of $k$-coloring $P_t$-free
graphs is not known  $k=3$, $t \geq 8$.  In this paper, we consider the \textsc{list-$3$ coloring problem} for $P_t$-free graphs with no induced  1-subdivision of $ K_{1,s}$. We use $ SDK_s $ to denote the one-subdivision of $K_{1,s}$. 
The main result is the following:
\begin{theorem}
For all positive integers $s$ and $t$, the	\textsc{list-$3$ coloring problem} can be solved in polynomial time for the class of $(P_t,SDK_s)$-free
graphs.
\end{theorem}

\section{Preliminaries}

We need two theorems: the first one is the famous Ramsey Theorem \cite{Ramsey}, and the second is a result of Edwards \cite{edwards}:
\begin{theorem}[\cite{Ramsey}]
	\label{Ramsey}
	For each pair of positive integers $k$ and $l$, there exists an integer $R(k,l)$ such that every graph with at least $R(k,l)$ vertices contains a clique with at least $k$ vertices or an independent set with at least $l$ vertices.
\end{theorem}

\begin{theorem}[\cite{edwards}]
	\label{Edwards}
	Let $G$ be  a graph, and let $L$ be a list assignment for $G$ 
	such that $|L(v)|\leq 2$
	for all $v\in V(G)$. Then a coloring of $(G,L)$, or a
	determination that none exists, can be obtained in time
	$O(|V(G)|+|E(G)|)$.
\end{theorem}

Let $G$ be a graph with list assignment $L$. For $X \subseteq V(G)$ we denote by
$G|X$ the subgraph induced by $G$ on $X$,  by
$G \setminus X$ the graph $G|(V(G) \setminus X)$ and by
$(G|X,L)$ the list coloring problem where we restrict the
domain of the list assignment $L$ to $X$. 
For $v \in V(G)$ we write $N_G(v)$ (or $N(v)$ when there is no danger of confusion) to mean the set of vertices of $G$ that are adjacent to  $v$. For $X\subseteq V(G)$ we write $N_G(X)$ (or $N(X)$ when there is no danger of confusion) to mean $\bigcup_{v\in X}N(v)$.
We say that $D\subseteq V(G)$ is a \emph{dominating set} of $G$ if for every vertex $v\in G\setminus D$, $N(v)\cap D\neq \emptyset$. By Theorem~\ref{Edwards}, the following corollary immediately follows.
\begin{Corollary}
	\label{Edwards2}
	Let $G$ be  a graph, $L$ be a $3$-list assignment for $G$ and let $D$ be a dominating set of $G$. Then a coloring of $(G,L)$, or a
	determination that $(G,L)$ is not colorable, can be obtained in time
	$O(3^{|D|}(|V(G)|+|E(G)|))$.
\end{Corollary}
\begin{proof}
	For every coloring $c$ of  $(G|D,L)$, in time $O(|E(G)|)$ we can define a list assignment $L_c$ of $G$ as follows:  if  $ v \in D $ we set  $ L_c(v)=\{c(v)\} $ and if $ v \notin D  $ we can pick $u\in N(v)\cap D$ by the definition of a dominating set and set $ L_c(v)=L(v)\setminus c(u)$.  Let $\mathcal{L}$ =$\{L_c:$ $c$ is a coloring of  $(G|D,L) \}$, then clearly $|\mathcal{L}|\leq 3^{|D|}$ and  $(G,L)$ is colorable if and only if there exists a $L_c\in \mathcal{L}$ such that $(G,L_c)$ is colorable. For every $L_c\in  \mathcal{L}$, by construction $ |L_c(v)|\leq 2 $ for every $v\in G$ and hence by Theorem~\ref{Edwards},  a coloring of $(G,L_c)$, or a	determination that none exists, can be obtained in time
	$O(|V(G)|+|E(G)|)$. Therefore a coloring of $(G,L)$, or a
	determination that $(G,L)$ is not colorable, can be obtained in time
	$O(3^{|D|}(|V(G)|+|E(G)|))$.
	
\end{proof}

\section{The Algorithm}
Let $s$ and $t$ be positive integers, and let $G=(V,E)$ be a connected
$ (P_t,SDK_s,K_4) $-free graph. Pick an arbitrary vertex $a\in V$ and let
$S_1=\{a\}$. For $v\in V$, let $d(v)$ be the distance from $v$ to $a$. For
$i=1,2,\dots, t-2$, we define the set $S_{i+1}$ as follows:
\begin{itemize}
	\item Let $B_i=N(S_i), W_i=V\setminus(B_i\cup S_i)$.
	\item Write  $S_i=\{v_1,v_2,\dots,v_{|S_i|}\}$ and define
           $$B_i^j=\left\{v\in \left(B_i\setminus \bigcup_{k=1}^{j-1}B_i^k\right) : v \textnormal{ is adjacent to } v_j  \right\}$$ for $j=1,2,\dots |S_i|.$
Then          $B_i=\bigcup_{j=1}^{|S_i|}B_i^j$.
	\item For $j=1,2,\dots, |S_i|$, let $X^j_i\subseteq B_i^j$ be a minimal vertex set such that for every $w\in W_i$, if $N(w)\cap B_i^j\neq \emptyset$, then  $N(w)\cap X_i^j\neq \emptyset$. Let $X_i=\bigcup_{j=1}^{|S_i|}X_i^j$.
	\item Let $S_{i+1}=S_i\cup X_i$.
\end{itemize}

It is clear that we can compute $S_{t-1}$ in $O(t|V|^2)$ time. Next, we prove some properties of this construction.
\begin{lemma}\label{size}
	For $i=1,2,\dots, t-2$, $|S_{i+1}|\leq |S_i| (1+R(4,R(4,s)))$.
\end{lemma}
\begin{proof}
	It is sufficient to show that for each $j=1,2,\dots, |S_i|$, $ |X^j_i| \leq R(4,R(4,s))$. Suppose not, $ |X^\ell_i |=K>R(4,R(4,s))$ for some $\ell\in \{1,2\dots, |S_i|\}$. Let $ X^\ell_i =\{x_1,x_2,\dots,x_K\}$. By the minimality of $X^\ell_i $, for $j=1,2,\dots, K$, there exists $y_j\in W_i$ such that  $N(y_j)\cap X^\ell_i=\{x_j\}$. Since $G$ is $K_4$-free, by Theorem \ref{Ramsey}, there exists a stable set $X'\subseteq X^\ell_i$  of size $R(4,s)$. We may assume $X'=\{x_1,x_2,\dots,x_{R(4,s)}\}$. Let $Y'=\{y_1,y_2,\dots,y_{R(4,s)}\}$. Again by Theorem \ref{Ramsey}, there exists a stable set $Y''\subseteq Y'$  of size $s$. We may assume $Y''=\{y_1,y_2,\dots,y_s\}$ and let $X''=\{x_1,x_2,\dots,x_s\}$. Then $G[\{v_\ell\}\cup X''\cup Y'']$ is isomorphic to $ SDK_s $, a contradiction.
\end{proof}

\begin{lemma}\label{length}
	For $i=0,1,2,\dots, t-2$, $B_{i+1} \setminus (B_i \cup S_i) = \{v : d(v) = i+1\}$ (where $S_0 = \emptyset$, $B_0 = \{a\}$ and $B_{t-1} = N(S_{t-1})$). 
\end{lemma}
\begin{proof}
	We use induction to prove this lemma. It is clear that for $i=0$, $B_1 = N(a) = \{v : d(v) = 1\}$. 
	
	Now suppose this lemma holds for $i < k$, where $k\in \{1,2\dots, t-2\}$. First we show that for every $v\in B_{k+1}\setminus (B_k \cup S_k)$, $d(v) = k+1$. By construction $v\in W_k$, hence $d(v) > k$ by induction. Since $v \in B_{k+1} \setminus B_k$, $v$ has a neighbor $w$ in $S_{k+1} \setminus S_k \subseteq B_k$; and thus $d(v) \leq d(w) + 1 \leq k+1$.
	
	Now let $v \in V$ with $d(v) = k+1$. It follows that $v \not\in (B_k \cup S_k)$, and $v \in B_{k+1} \cup W_{k+1}$, and $v$ has a neighbor $w \in V$ with $d(w) = k$. By induction, it follows that $v \in W_k$ and $w \in B_k$. Let $j \in \mathbb{N}$ such that $w \in B_k^j$. Since $v \in W_k$ and $N(w) \cap B_{k}^j \neq \emptyset$, it follows that $v$ has a neighbor in $X_k^j \subseteq X_k \subseteq S_{k+1}$, and therefore $v \in B_{k+1}$, as required.  This finishes the proof of Lemma~\ref*{length}.
\end{proof}

By applying Lemma~\ref{size} and Lemma~\ref{length}, we deduce several properties of $S_{t-1}$.
\begin{lemma}\label{S}
	\begin{enumerate}
		\item There exists a constant $M_{s,t}$ which only depends on $s$ and $t$ such that $|S_{t-1}|\leq M_{s,t}$.
		\item $W_{t-1}=V\setminus (S_{t-1}\cup N(S_{t-1}))=\emptyset$.
	\end{enumerate}
\end{lemma}
\begin{proof}
	Since we start with $|S_1|=1$,  by applying Lemma~\ref{size} $ t-2 $ times, it follows that $|S_{t-1}|\leq (1+R(4,R(4,s)))^{t-2}$. Let $M_{s,t}=(1+R(4,R(4,s)))^{t-2}$, then the first claim holds.
	
	Suppose the second claim does not hold. From Lemma \ref{length}, it follows that $\{v : d(v) \leq t-1\} \subseteq S_{t-1} \cup N(S_{t-1})$. But if $w \in V$ satisfies $d(w) \geq t-1$, then a shortest $w$-$a$-path is an induced path of length at least $t$, a contradiction. Thus the second claim holds.
\end{proof}	

We are now ready to prove our main result, which we rephrase here:

\begin{theorem}
	Let $ M_{s,t}=(1+R(4,R(4,s)))^{t-2} $. There exists an algorithm with running time $O(|V(G)|^4+t|V(G)|^2+3^{M_{s,t}}(V(G)+E(G)))$ with the following specification. 
	\\
	\\
	{\bf Input:}  A $(P_t,SDK_s)$-free graph G and a $3$-list assignment $L$ for $G$.
	\\
	\\
	{\bf Output:}  A coloring of $(G,L)$, or a determination that $(G,L)$ is not colorable.
\end{theorem}
\begin{proof}
	We may assume that $G$ is connected, since otherwise we can run the algorithm for each component of $G$ independently. In time $O(|V(G)|^4)$ we can determine that either $(G,L)$ is not colorable, or $G$ is $K_4$-free. If $G$ is $K_4$-free, we can construct $S_{t-1}$ in $O(tn^2)$ time as stated above. Then by Lemma~\ref{S}, $S_{t-1}$ is a dominating set of $G$ and $|S_{t-1}|\leq M_{s,t}$. Now the theorem follows from Corollary~\ref{Edwards2}.
\end{proof}	

\end{document}